\documentclass[12pt,leqno]{amsart}
\usepackage[
margin=1.5in,
]{geometry}
\usepackage{tikz,tikz-cd}
\usepackage{url}
\usepackage{graphicx}
\usepackage{graphicx}
\usepackage{color}
\usepackage{times}
\usepackage{cite}
\usepackage{enumerate,latexsym}
\usepackage{latexsym}
\usepackage{amsmath,amssymb}
\usepackage{graphicx}
\usepackage{amsthm}
\usepackage{verbatim}
\newtheorem{thm}{Theorem}[section]
\newtheorem*{theorem*}{Theorem}
\newtheorem*{acknowledgement*}{Acknowledgement}

\newtheorem{lem}[thm]{Lemma}
\newtheorem{prop}[thm]{Proposition}
\theoremstyle{definition}

\theoremstyle{remark}
\newtheorem{rem}[thm]{Remark}

\theoremstyle{definition}

\numberwithin{equation}{section}

\newcommand{\set}[1]{\left\{#1\right\}}
\newcommand{\Real}{\mathbb R}

\newcommand{\dist}[0]{\mathrm{dist}}

\title{ Rigidity Properties of Colding-Minicozzi Entropies}

	\author{Jacob Bernstein}
\address{Department of Mathematics, Johns Hopkins University, 3400 N. Charles Street, Baltimore, MD 21218}
\email{bernstein@math.jhu.edu}
\thanks{The author was partially supported by the  DMS-2203132 and the Institute for Advanced Study with funding provided by the Charles Simonyi Endowment.}

\begin{document}
\dedicatory{Dedicated to Joel Spruck in honor of his retirement.} 
\begin{abstract}
We show certain rigidity for minimizers of generalized Colding-Minicozzi entropies.  The proofs are elementary and work even in situations where the generalized entropies are not monotone along mean curvature flow.
\end{abstract}
\maketitle

\section{Introduction}
We use an expansion of the volume of a submanifold in small geodesic balls as in \cite{karpVolumeSmallExtrinsic1989} to show some rigidity phenomena for natural generalizations of the Colding-Minicozzi entropy \cite{Coldinga, BernsteinBhattacharyaCartan}.   In particular, the arguments work even when the quantities are not monotone along mean curvature flow.  

In order to define the generalized entropies we begin by setting
\begin{equation}\label{KnkappaEqn}
	K_{n,\kappa}(t,r)= \left\{\begin{array}{cc} \kappa^nK_n(\kappa^2 t, \kappa r) & \kappa,t>0,  r\geq 0\\
		(4\pi t)^{-n/2} e^{-\frac{r^2}{4 t}} & \kappa=0, t>0, r\geq  0
	\end{array}\right.
\end{equation}
where $n\geq 1$,  $\kappa\geq 0$ and $K_n$ are explicit (if complicated) functions used in \cite{daviesHeatKernelBounds1988} to study the heat kernel on hyperbolic space. For instance, 
$$
K_3(t,r)= (4\pi t)^{-\frac{3}{2}} \frac{r}{\sinh(r)}e^{-t-\frac{r^2}{4t}}.$$
The other $K_n$ are determined recursively -- see \cite{daviesHeatKernelBounds1988} for details.  In general,  
$$
H_{n, \kappa}(t,x; t_0, x_0)=K_{n,\kappa}(t-t_0, \dist_{g}(x,x_0))
$$
is the heat kernel on $(M,g)$ with singularity at $x_0$ and time $t_0$ precisely when $(M,g)$ is a simply connected space form of constant curvature $-\kappa^2$.  Following \cite{Coldinga,BernsteinHypEntropy, BernsteinBhattacharyaCartan}, let
$$
\Phi_{n, \kappa}^{t_0,x_0}(t,x)=K_{n,\kappa}(t_0-t, \dist_{g}(x,x_0))
$$
and for $\Sigma\subset M$, an $n$-dimensional submanifold,  define \emph{the Colding-Minicozzi $\kappa$-entropy of $\Sigma$ in $(M,g)$} to be
$$
\lambda_g^\kappa[\Sigma]=\sup_{x_0\in M, \tau>0}\int_{\Sigma} \Phi_{n, \kappa}^{0, x_0} (-\tau, \cdot) dV=\sup_{x_0\in M, \tau>0}\int_{\Sigma} \Phi_{n, \kappa}^{\tau, x_0} (0, \cdot) dV.
$$
When $\kappa=0$ and $(M,g)=(\Real^{n+k}, g_{\Real})$ is Euclidean space, this is the usual Colding-Minicozzi entropy, $\lambda[\Sigma]$, of $\Sigma$ from \cite{Coldinga}. When $\kappa=1$ and $(M,g)=(\mathbb{H}^{n+k}, g_{\mathbb{H}})$ is hyperbolic space, it is the entropy in hyperbolic space, $\lambda_{\mathbb{H}}[\Sigma]$, from \cite{BernsteinHypEntropy}.

In \cite[Theorem 1]{BernsteinBhattacharyaCartan}, it is shown that if $(M,g)$ is an $(n+k)$-dimensional Cartan-Hadamard manifold with $\sec_{g}\leq -{\kappa}^2_0$ and $0\leq \kappa \leq \kappa_0$, then, for any mean curvature flow of closed $n$-dimensional submanifolds, $t\in [t_1,t_2]\mapsto \Sigma_{t}\subset M$, 
$$
\lambda_{g}^\kappa[\Sigma_{t_1}]\geq \lambda_{g}^\kappa[\Sigma_{t_2}].
$$
This generalizes and unifies the monotonicity properties of the entropy of \cite{Coldinga} and \cite{BernsteinHypEntropy}.  Monotonicity also holds for non-closed flow under appropriate hypotheses.

It follows readily from the definition that for any $n$-dimensional submanifold $\Sigma\subset M$, one has $\lambda_g^\kappa[\Sigma]\geq 1$ -- see \cite[Proposition 6.3]{BernsteinBhattacharyaCartan}.  We seek to understand what can be said about $\Sigma$ in the case of equality -- i.e., when $\Sigma$ minimizes a Colding-Minicozzi $\kappa$-entropy.  The monotonicity of entropy can be used to answer this and related problems -- see \cite{chenRigidityStabilitySubmanifolds2021, BernsteinWang1, JZhu}.  However, the question makes perfectly good sense in arbitrary Riemannian manifolds where monotonicity may not hold.  With this in mind, we establish some rigidity properties that hold without using monotonicity.

\begin{thm}\label{SurfaceRigidityThm}
	Let $(M,g)$ be a $(2+k)$-dimensional Riemannian manifold with $sec_g\leq -\kappa^2$.  If $\Sigma$ is a proper surface and $\lambda_g^\kappa[\Sigma]=1$, then $\Sigma$ is totally umbilic.  In particular, if $(M, g)$ is Euclidean space and $\kappa=0$, then $\Sigma$ is an affine two-plane.
\end{thm}
\begin{rem}
	The result in Euclidean space follows from earlier work of L. Chen \cite{chenRigidityStabilitySubmanifolds2021} who was also able to obtain the same result for all dimensions and co-dimensions and also for incomplete surfaces.  However, his argument  uses a fairly sophisticated mean curvature flow construction. Alternatively, it should also be possible to obtain rigidity  for hypersurfaces that are boundaries of nice enough subsets of $\Real^{n+1}$ using the Gaussian isoperimetric inequality \cite{sudakovExtremalPropertiesHalfspaces1978, borellEhrhardInequality2003}.
\end{rem}
 The techniques also allow us to show rigidity for minimal hypersurfaces in Einstein manifolds with appropriate Einstein constant -- unlike the preceding theorem this is a setting where one may not have monotonicity of the entropy.
\begin{thm}\label{EinsteinRigidityThm}
	Let $(M,g)$ be an $n$-dimensional Riemannian manifold satisfying
	$$
	Ric_g=-(n-1) \kappa^2 g.
	$$
 If $\Sigma$ is a proper minimal hypersurface with $\lambda_g^\kappa[\Sigma]=1$, then $\Sigma$ is  totally geodesic.  
\end{thm}
Finally, we obtain a result for certain closed surfaces with $\lambda_g^0[\Sigma]=1$ without any assumptions on the ambient manifold.
\begin{thm}\label{UnivRigidityThm}
	Let $(M,g)$ be a $(2+k)$-dimensional Riemannian manifold.  If $\Sigma$ is a closed, connected,  orientable surface satisfying $\lambda_g^0[\Sigma]=1$, then the genus of $\Sigma$ satisfies $\mathrm{gen}(\Sigma)\leq 1$.  Moreover, if $\mathrm{gen}(\Sigma)=1$, then $\Sigma$ is totally geodesic.
\end{thm}
\begin{rem}
	This result is sharp in a certain sense as can be seen by considering a totally geodesic flat two-torus inside of a higher dimensional flat torus.
\end{rem}

We note that similar, but stronger, rigidity results for conformal volume of submanifolds of the sphere were observed by Bryant in \cite{bryantSurfacesConformalGeometry1988}.  One difference between \cite{bryantSurfacesConformalGeometry1988} and the current paper is that it is possible to arbitrarily  change the mean curvature at a point with a conformal transformation -- i.e., with a symmetry of the conformal volume.  For Colding-Minicozzi entropies this cannot be done with the natural symmetries.  However, flowing by mean curvature flow seems to play a similar role.

%

\section{Small-time asymptotics of Gaussian $\kappa$-densities of submanifolds}
Let $(M,g)$ be a Riemannian manifold and $\Sigma\subset M$ a proper $n$-dimensional submanifold.
We obtain the small time asymptotics of the (localized) pairing of the kernel $\Phi_{n,\kappa}^{0,x_0}(-t, \cdot)$  with any $n$-dimensional submanifold of $M$ when $x_0\in \Sigma$. 
 \begin{prop} \label{ShortTimeProp}
 	Let $(M,g)$ be a $(n+k)$-dimensional Riemannian manifold and $\Sigma^n \subset M$ a $n$-dimensional submanifold.  For $x_0\in \Sigma$, if $\mathcal{B}^g_{2R}(x_0)$ is proper in $M$ and $\Sigma_{2R}=\mathcal{B}^g_{2R}(x_0)\cap \Sigma$ is proper in $\mathcal{B}^g_{2R}(x_0)$,  then, for, any $\kappa\geq 0$,  the following asymptotic expansion holds:
	\begin{align*}
		\int_{\Sigma_{R}} \Phi_{n,\kappa}^{0, x_0} (t, \cdot )dV &= 1 - \frac{t}{3}  \left(\frac{1}{2}|\mathbf{A}_\Sigma^g(x_0)|^2+\frac{1}{4}|\mathbf{H}_\Sigma^g(x_0)|^2-R^g_\Sigma(x_0)-n(n-1) \kappa^2\right)\\
		&+O\left((-t)^{\frac{3}{2}}\right), t\to 0^-.
	\end{align*}
Where here $\mathbf{A}_\Sigma^g$ is the second fundamental form of $\Sigma$, $\mathbf{H}^g_{\Sigma}=\mathrm{tr}_g \mathbf{A}_{\Sigma}^g$ is the mean curvature vector of $\Sigma$ and $R_\Sigma^g$ is the scalar curvature of $\Sigma$.
\end{prop}

In order to prove this, we will need a pair of elementary lemmata.
\begin{lem}\label{GaussianIntLem}
  For any $R>0$ and $n\geq 1, k\geq 0$ integers one has
  $$ 
  	\int_0^R K_{n,0}(t, \rho) \rho^{n+k-1} d\rho =\frac{(4\pi t)^{k/2}}{|\mathbb{S}^{n+k-1}|_{\Real}}+O(e^{-\frac{R^2}{4t}}), t\to 0^+.
  $$
 There are also constants $C=C(R,n)$ so, for $0<t\leq \frac{1}{2(n-1)}R$,
  $$
  \int_{R}^\infty K_{n,1}(t, \rho) \sinh^{n-1}(\rho)d\rho \leq C e^{-\frac{R^2}{16 t}}.
  $$
\end{lem}
\begin{proof}
	Observe that, for all $k\geq 0$, 
	\begin{align*}
		\int_0^\infty K_{n,0}(t, \rho) \rho^{n+k-1} d\rho =(4\pi t)^{k/2} \int_0^\infty K_{n+k,0}(t,\rho) \rho^{n+k-1} d\rho =\frac{(4\pi t)^{k/2}}{|\mathbb{S}^{n+k-1}|_{\Real}}.
	\end{align*}
	While for any $R>0$ and for small times we have
	$$
	\int_R^\infty K_{n,0} (t,\rho) \rho^{n+k-1} d\rho=O(e^{-\frac{R^2}{4t}}), t\to 0^+.
	$$
The first claim is an immediate consequence.	

For the second claim we observe that \cite[Theorem 3.1]{daviesHeatKernelBounds1988} yields
$$
K_{n,1}(t,\rho)\leq C (1+\rho+t)^{\frac{1}{2}n-\frac{3}{2}}(1+\rho)e^{-\frac{1}{2} (n-1)^2 t-\frac{1}{2}(n-1)\rho} K_{n,0}(t,\rho).
$$
One readily checks that, for $\rho\geq R>0$ and $ \frac{R}{2(n-1)}\geq t$ that, 
\begin{align*}
K_{n,1}(t,\rho)& \sinh^{n-1}(\rho) \leq C' (1+\rho+t)^{\frac{1}{2}n-\frac{3}{2}}(1+\rho)K_{n,0}(t, \rho-t(n-1))\\
&\leq C'' R^{-\frac{1}{2}(n-1)} (\rho-t(n-1))^{n-1} K_{n,0}(t,\rho-t(n-1)). 
\end{align*}
where $C''=C''(R,n)$. 
It follows that, for such $R$ and $t$,
\begin{align*}
 \int_{R}^\infty &K_{n,1}(t, \rho) \sinh^{n-1}(\rho)d\rho \leq C'' R^{-\frac{1}{2}(n-1)} \int_{R-t(n-1)}^\infty  u^{n-1} K_{n,0}(t,u)du\\
 &\leq C'' R^{-\frac{1}{2}(n-1)} \int_{\frac{R}{2}}^\infty  u^{n-1} K_{n,0}(t,u)du\leq C''' e^{-\frac{R^2}{16 t}}.
\end{align*}
Here $C'''=C'''(R,n)$ and we used the first computation of the proof. 
\end{proof}

We also need information about the leading order asymptotics of $K_{n,\kappa}$ near the the space-time origin -- the expansions was established in \cite{daviesHeatKernelBounds1988}, but we needed some more information about the relationship between certain coefficients. 
\begin{lem}
Fix $\kappa>0$. 	There is a constant $C_n>0$ so that, for $0\leq t\leq \kappa^{-2}$ and $0\leq \rho \leq \kappa^{-1}$,
\begin{align*}
	|K_{n,\kappa}(t,\rho) -(1+\kappa^2 a_n t+\kappa ^2 b_n \rho^2) K_{n,0}(t,\rho)|\leq C_n\kappa^4(t+\rho^2)^2 K_{n,0}(t,\rho)
\end{align*}
where $a_n,b_n$ satisfy
\begin{equation}\label{CoeffFormula}
a_n+2n b_n=-\frac{1}{3}n(n-1). 
\end{equation}
\end{lem}
\begin{proof}
	With out loss of generality we may assume $\kappa=1$, as the general result immediately follows from this case and the definition of $K_{n,\kappa}$.  The existence of the $a_n, b_n$ and $C_n$ follow from \cite{daviesHeatKernelBounds1988}.  To conclude the proof we observe that for, all $t>0$,
	\begin{align*}
		|\mathbb{S}^{n-1}|_\Real \int_0^\infty K_{n,1}(t,\rho) \sinh^{n-1}(\rho) d\rho = 1.
	\end{align*}

 Using the second estimate of Lemma \ref{GaussianIntLem} with $R=1$ and  the expansion $\sinh^{n-1}(\rho)=\rho+\frac{n-1}{6}\rho^3+O(\rho^5)$,  we obtain
   \begin{align*}
   	 |\mathbb{S}^{n-1}|_{\Real}^{-1}&= \int_0^1 K_{n,1}(t,\rho) \sinh^{n-1}(\rho) d\rho+O(t^2),t\to 0^+\\
   	 &= \int_0^1 K_{n,0}(t, \rho)(1+a_n t+b_n \rho^2+\frac{n-1}{6}\rho^2) \rho^{n-1} d\rho+O(t^2),t\to 0^+\\
   	 &=|\mathbb{S}^{n-1}|_{\Real}^{-1}\left( 1+\right(a_n+2 n  b_n +\frac{1}{3} n(n-1)\left) t\right)+O(t^2),t\to 0^+. 
\end{align*}
where we used that 
\begin{equation}\label{SphereVolEqn}
|\mathbb{S}^{n+k+1}|_{\Real}=\frac{2\pi}{n+k} |\mathbb{S}^{n+k-1}|_{\Real}.
\end{equation}
Hence, $
a_n+2n b_n=-\frac{1}{3}n(n-1). $
%
\end{proof}

\begin{proof}[Proof of Proposition \ref{ShortTimeProp}]
For the fixed point $x_0\in \Sigma$, choose $R\geq R_0>0$ small enough so that $\mathcal{B}_{R_0}^g(x_0)$ is geodesically convex.  Up to shrinking $R_0$, we may assume that the expansion of Theorem \ref{GenKarpPinksy} holds for $|\mathcal{B}_s^g(x_0)\cap \Sigma|_g$ for $0<s<R_0$.  As $\Sigma_R$ is proper in $\mathcal{B}^g_{2R}(\Sigma)$, we have $|\Sigma_R|_g$ finite.  Hence, by the pointwise estimates on $K_{n,\kappa}$ that follow from \cite[Theorem 3.1]{daviesHeatKernelBounds1988}, we have, for $-t$ sufficiently small,
$$
\int_{\Sigma_R \setminus \mathcal{B}^g_{R_0}(x_0)} \Phi_{n,\kappa}^{0, x_0}(-t, \cdot) dV\leq C |\Sigma_R|_g e^{-\frac{R_0^2}{-16 t}}=O((-t)^{\frac{3}{2}}).
$$
where here $C=C(n,\kappa)$.

Hence, it is enough to prove the result for $\Sigma_{R_0}=\Sigma\cap \mathcal{B}_{R_0}^g(x_0)$. 
Using the co-area formula we have, with $\rho(q)=\dist_g(q,x_0)$,
 \begin{align*}
 	\int_{\Sigma_{R_0}}\Phi_{n,\kappa}^{0,x_0} (t, \cdot) dV &= \int_0^{R_0} \int_{\partial \mathcal{B}_s^g (x_0)} \Phi_{n,\kappa}^{0,x_0} (t, \cdot) \frac{1}{|\nabla_\Sigma \rho|} dV ds &\\
 	&= \int_0^{R_0} K_{n,\kappa}(-t,s) \int_{\partial \mathcal{B}_{s}^g (x_0)}  \frac{1}{|\nabla_\Sigma \rho|} dV ds \\
 	&=\int_0^{R_0} K_{n,\kappa}(-t,s) \frac{d}{ds}|\mathcal{B}_s^g(x_0)\cap \Sigma| ds \\
 	&=K_{n,\kappa}(-t, R_0) |\Sigma_{R_0}|_g-\int_0^\infty \partial_r K_{n,\kappa}(-t,s) |\mathcal{B}_s^g(x_0)\cap \Sigma| ds.
 \end{align*}
Appealing again to \cite[Theorem 3.1]{daviesHeatKernelBounds1988}, we have
$$
K_{n,\kappa}(-t, R_0) |\Sigma_{R_0}|_g\leq  C |\Sigma_{R}|_g e^{-\frac{R_0^2}{-16 t}}=O((-t)^{\frac{3}{2}}).
$$

 Moreover, by direct computation we have 
 $$
 \partial_r K_{n,0}(-t,r)=-\frac{r}{2t}K_{n,0}(-t,r)=-2\pi r K_{n+2,0}(-t,r)
 $$
 and, for $\kappa>0$, the generalized Millison identity (e.g., \cite{daviesHeatKernelBounds1988, BernsteinBhattacharyaCartan}) give
 $$
 \partial_r K_{n,\kappa}(-t,r)=-2\pi e^{-n\kappa^2t}\kappa^{-1} \sinh(\kappa r) K_{n+2,\kappa r}(-t,r).
 $$
 Hence, when $\kappa=0$, we may use Theorem \ref{GenKarpPinksy} and Lemma \ref{GaussianIntLem} to obtain
 \begin{align*}
 	\int_{\Sigma_{R_0}}&\Phi_{n,\kappa}^{0,x_0} (-t, \cdot) dV= 2\pi \int_0^{R_0} K_{n+2,0}(-t,s) s |\mathcal{B}^g_s(x_0)\cap \Sigma|ds+O((-t)^{\frac{3}{2}}) \\
 	&= 2\pi |B_1^n|_\Real\int_0^{R_0} K_{n+2,0}(-t,s) \left( s^n+A s^{n+1} +O(s^{n+2})\right) ds+O((-t)^{\frac{3}{2}})\\
 	&=  |\mathbb{S}^{n+1}|_\Real\left( |\mathbb{S}^{n+1}|_{\Real}^{-1} +4\pi |\mathbb{S}^{n+3}|_{\Real}^{-1} A (-t) \right) +O((-t)^{\frac{3}{2}})\\
 	&= 1-2(n+2) A t+ O((-t)^{\frac{3}{2}}).
 \end{align*}
Where we used \eqref{SphereVolEqn} and the coefficient $A$ from Theorem \ref{GenKarpPinksy} is 
$$
A=\frac{1}{6(n+2)}\left( \frac{1}{2}|\mathbf{A}_\Sigma^g(p)|_g^2 +\frac{1}{4}|\mathbf{H}_\Sigma^g(p)|^2_g -R_{\Sigma}^g(p) \right).
$$
The expansion in the $\kappa=0$ case follows.  
 
When $\kappa>0$ the same reasoning yields 
 \begin{align*}
  	\int_{\Sigma_{R_0}}&\Phi_{n,\kappa}^{0,x_0} (-t, \cdot) dV= 2\pi e^{-n \kappa^2 t} \kappa^{-1} \int_0^{R_0} K_{n+2,\kappa}(-t,s) \sinh(\kappa s) |\mathcal{B}_s^g(x_0)\cap \Sigma|_g ds   \\
  	&= 	2\pi |B_1^n|_{\Real}e^{-n \kappa^2 t} \int_0^{R_0} K_{n+2,\kappa}(-t,s) ( s+ \frac{\kappa^2}{6}s^3 ) (s^{n}+ As^{n+2}+O(s^{n+3}))) ds \\
  	&= |\mathbb{S}^{n+1}|_{\Real} e^{-n \kappa^2 t}  \int_0^{R_0} K_{n+2, \kappa}(-t,s) (1+As^2 +\frac{\kappa^2}{6}s^2+O(s^3)) s^{n+1} ds.
\end{align*}
We now apply Lemma \ref{GaussianIntLem} ignoring terms of order $O((-t)^\frac{3}{2})$ to obtain
\begin{align*}
   & |\mathbb{S}^{n+1}|_{\Real} e^{-n \kappa^2 t}  \int_0^{R_0} K_{n+2, 0}(-t,s) \left(1+As^2 +\kappa^2\left(\frac{s^2}{6} + b_{n+2}s^2- a_{n+2} t\right)\right)ds \\
    	&= e^{-n \kappa^2 t} \left( 1-t \left(2(n+2) A+\kappa^2\left(a_{n+2}+2(n+2)b_{n+2} +\frac{n+2}{3}\right)\right)\right) \\
    &=1-t\left(2(n+2){-2} A+\kappa^2\left(a_{n+2}+2(n+2)b_{n+2} +\frac{n+2}{3}+n\right)\right) \\
  	&=1-t\left(2(n+2) A-\frac{1}{3}n(n-1)\kappa^2 \right).
   \end{align*}
Where the last equality uses Lemma \ref{CoeffFormula}. The result is immediate.
\end{proof}

\section{Rigidity}
We are now able to prove the main rigidity results of the paper. 

\begin{proof}[Proof of Theorem \ref{SurfaceRigidityThm}]
 It follows from the Gauss equations and the hypothesis $sec_g\leq -\kappa^2$  that, for any $x_0\in \Sigma$, 
 $$
  -R_{\Sigma}^g(x_0) \geq  n(n-1) \kappa^2 +|\mathbf{A}_\Sigma^g |^2-|\mathbf{H}_\Sigma^g|^2.
 $$
Hence,  Proposition \ref{ShortTimeProp} and $\lambda_g^\kappa[\Sigma]=1$ together imply that, for any $x_0\in \Sigma$, 
\begin{align*}
 0&\geq \frac{1}{2}|\mathbf{A}_\Sigma^g(x_0)|^2+\frac{1}{4}|\mathbf{H}_\Sigma^g(x_0)|^2-n(n-1) \kappa^2-R^g_\Sigma(x_0)\\
 &\geq \frac{3}{2} |\mathbf{A}_\Sigma^g(x_0)|^2-\frac{3}{4}|\mathbf{H}_\Sigma^g(x_0)|^2= \frac{3}{2} |\mathring{\mathbf{A}}_{\Sigma}^g (x_0)|^2.
\end{align*}
Where the last equality used that $\Sigma$ was two dimensional and $\mathring{\mathbf{A}}_{\Sigma}^g= \mathbf{A}_\Sigma^g-\frac{1}{2} \mathbf{H}_{\Sigma}^g g_{\Sigma}$ is the trace-free part of the second fundamental form of $\Sigma$.
As $x_0$ was arbitrary, we conclude that $\mathring{\mathbf{A}}_\Sigma^g$ vanishes and so $\Sigma$ is totally umbilic.

To conclude the proof we observe that if the ambient space is Euclidean and $\kappa=0$, then, as  $\Sigma$ is proper, it must be a collection of affine two-planes and  round two-spheres contained in three-dimensional affine subspaces.  However, one can readily compute that any such two-sphere has  entropy strictly larger than 1. Likewise, if there is more than one affine two-plane the entropy is also strictly larger than 1 and so the claim follows.
\end{proof}
We again use the Gauss equations to obtain the rigidity of Theorem \ref{EinsteinRigidityThm}.
\begin{proof}[Proof of Theorem \ref{EinsteinRigidityThm}]
The Gauss equations imply that if $\Sigma$ is a $n$-dimensional hypersurface in $M$, then 
$$
-R_{\Sigma}^g=-R_g +2Ric_g(\mathbf{n}, \mathbf{n})+|\mathbf{A}_\Sigma^g|^2-|\mathbf{H}_\Sigma^g|^2.
$$
The Einstein condition on $(M,g)$ gives $-R_g=(n+1)n \kappa^2$ and so
$$
-R_{\Sigma}^g=n(n-1) \kappa^2+|\mathbf{A}_\Sigma^g|^2-|\mathbf{H}_\Sigma^g|^2.
$$
Hence, the hypotheses that $\lambda_g^\kappa[\Sigma]=1$ together with Proposition \ref{ShortTimeProp} implies that, for any $x_0\in \Sigma$, 
\begin{align*}
0 & \geq \frac{1}{2}|\mathbf{A}_\Sigma^g(x_0)|^2+\frac{1}{4}|\mathbf{H}_\Sigma^g(x_0)|^2-n(n-1) \kappa^2-R^g_\Sigma(x_0)\\
&= \frac{3}{2} |\mathbf{A}_\Sigma^g(x_0)|^2-\frac{3}{4}|\mathbf{H}_\Sigma^g(x_0)|^2= \frac{3}{2} |\mathbf{A}_\Sigma^g(x_0)|^2.
\end{align*}
Here the last equality used that $\Sigma$ was minimal.
Hence, as $x_0$ was arbitrary, $\Sigma$ is totally geodesic.
\end{proof}

Finally, we use the Gauss-Bonnet theorem to prove Theorem \ref{UnivRigidityThm}.
\begin{proof}[Proof of Theorem \ref{UnivRigidityThm}]
As above, $\lambda_g^0[\Sigma]=1$ and Proposition \ref{ShortTimeProp} together imply that
$$
0\geq \frac{1}{2}|\mathbf{A}_\Sigma^g|^2 +\frac{1}{4}|\mathbf{H}_\Sigma^g|^2-R_{\Sigma}^g
$$
for all points on $\Sigma$.  As $\Sigma$ is closed,  we may integrate this inequality over $\Sigma$ to obtain
$$
\int_{\Sigma} R_{\Sigma}^g dV\geq \frac{1}{2}\int_\Sigma |\mathbf{A}_\Sigma^g|^2 +\frac{1}{2}|\mathbf{H}_\Sigma^g|^2 dV\geq 0.
$$
As $\Sigma$ is closed and orientable, the Gauss-Bonnet theorem implies
$$
8\pi (1-\mathrm{gen}(\Sigma))\geq \frac{1}{2}\int_\Sigma |\mathbf{A}_\Sigma^g|^2 +\frac{1}{2}|\mathbf{H}_\Sigma^g|^2 dA\geq 0.
$$
Hence, $\mathrm{gen}(\Sigma)\leq 1$ and and if $\mathrm{gen}(\Sigma)=1$, then $\Sigma$ is totally geodesic.
\end{proof}


\appendix
\section{Generalized Karp-Pinksy Expansion}

We record here an asymptotic expansion of the volume of a submanifolds inside a geodesic ball in some Riemannian manifold.  This generalizes a result of Karp and Pinsky \cite{karpVolumeSmallExtrinsic1989} who treated the Euclidean case.  Such an expansion is shown for submanifolds of hyperbolic space in \cite{carrerasVolumeSmallExtrinsic1998} -- see also \cite{grayVolumeSmallGeodesic1974,grayRiemannianGeometryDetermined1979}.  

\begin{thm}\label{GenKarpPinksy}
	Let $(M,g)$ be a Riemannian manifold and $\Sigma \subset M$ a submanifold of dimension $n$.  For $p\in \Sigma$, one has the following expansion for $R>0$ small
	\begin{align*}
		|\mathcal{B}_R^g(p)\cap \Sigma|_g &= |B_1^n|_{\mathbb{R}} R^n +A|B_1^n|_{\Real} R^{n+2} +O(R^{n+3}),
\end{align*}
here $|\cdot |_g $ denote the $g$-volume and $|B_1^n|_{\Real}$ is the Euclidean volume of the ball and
\begin{align*}
		A&= \frac{1}{6(n+2)}\left( \frac{1}{2}|\mathbf{A}_\Sigma^g(p)|_g^2 +\frac{1}{4}|\mathbf{H}_\Sigma^g(p)|^2_g -R_{\Sigma}^g(p) \right),
	\end{align*}
where $\mathbf{A}_\Sigma^g$ is the second fundamental form,  $\mathbf{H}_\Sigma=\mathrm{tr}_g \mathbf{A}_\Sigma^g$ is the mean curvature vector and $R_\Sigma^g$ is the the scalar curvature of $\Sigma$ with its induced metric.
\end{thm}

We will prove this by isometrically embedding $M$ into a large Euclidean space.  First, we fix notation and suppose that $M$ is $(n+K)$-dimensional and $i: M\to \Real^{n+K+N}=\Real^{n+K}_{\mathbf{x}}\times \Real^N_{\mathbf{y}}$ is an isometric embedding with $i(p)=\mathbf{0}$ and so $di_p: T_p M\to \Real^n_{\mathbf{x}}\times \set{0}$.  In particular, if we set $M'=i(M)$, then $\mathbf{0}\in M'$ and $\Real^{n+K}\times \set{\mathbf{0}}=\set{\mathbf{y}=\mathbf{0}}=T_{\mathbf{0}} M'$  Let $g'$ be the induced metric on $M'$ -- i.e.,  so $i^*g'=g$.

We now define two maps from a neighborhood of $\mathbf{0}$ in $\Real^{n+K}$ to $M'$.  For the first,  observe that, near $\mathbf{0}$, $M'$ is the graph of a function $\mathbf{u}(\mathbf{x})=(u_1(\mathbf{x}), \ldots, u_{N}(\mathbf{x}))$.  Hence, the first map, $G_{M'}$, may be defined in a neighborhood, $U$, of $\mathbf{0}$ by
$$
G_{M'}: U\to M', \mathbf{x}\mapsto (\mathbf{x}, \mathbf{u}(\mathbf{x})).
$$
Note the hypotheses on $M'$ ensure $\mathbf{u}(\mathbf{0})=\mathbf{0}$ and $D\mathbf{u}(\mathbf{0})=\mathbf{0}$.  Hence,
$$
u_\alpha(\mathbf{x})=\frac{1}{2} \sum_{i,j=1}^{n+K} C^{ij}_\alpha x_i x_j+O(|\mathbf{x}|^3)
$$
where we can readily identify 
$$C^{ij}_\alpha=\mathbf{e}_{n+k+\alpha}\cdot \mathbf{A}_{M'}^\Real|_{\mathbf{0}}(\mathbf{e}_i, \mathbf{e}_j)$$
where $\mathbf{A}_{M'}^\Real$ is the second fundamental form of $M'$. Up to shrinking $U$, we may also define a second map based on the exponential map of $g'$.
$$
E: U\to M', \mathbf{x}\mapsto \exp^{g'}_{\mathbf{0}}(\mathbf{x})=i(\exp^g_p(\mathbf{x}))
$$
where here we think of $\mathbf{x}$ as an element of $T_\mathbf{0}M'$ and also identify it in the natural way with an element of $T_pM$ via the isomorphism $di_p: T_pM \to T_{\mathbf{0}}M'$. 

Let $\Sigma'$ be the $n$-dimensional submanifold of $M'$ so $i(\Sigma')=\Sigma$.  Write $\Real^{n+K}=\Real^n_{\mathbf{w}}\times \Real^K_{\mathbf{z}}$.   We have $\mathbf{0}\in \Sigma'$ and, up to rotating the $\Real^{n+K}$ factor, may assume $T_{\mathbf{0}} \Sigma'=\Real^n\times \set{\mathbf{0}}=\set{\mathbf{z},\mathbf{y}=\mathbf{0}}\subset \Real^{n+K+N}$.  Hence, near $\mathbf{0}$ we can express $\Sigma'$ as the graph of a function $\mathbf{v}(\mathbf{w}, \mathbf{z}(\mathbf{w})=(v_1(\mathbf{w}), \ldots, v_{K+N}(\mathbf{w}))$ and  define a map $G_{\Sigma'}$ in a neighborhood, $V$, of $\mathbf{0}$
$$
G_{\Sigma'}: V\to \Sigma', \mathbf{w}\mapsto (\mathbf{w}, \mathbf{v}(\mathbf{w})).
$$
Note the hypotheses on $\Sigma'$ ensure $\mathbf{v}(\mathbf{0})=\mathbf{0}$ and $D\mathbf{v}(\mathbf{0})=\mathbf{0}$.  In particular, we have
$$
v_\alpha(\mathbf{w})=\frac{1}{2} \sum_{i,j=1}^{n} A^{ij}_\alpha w_i w_j+O(|\mathbf{w}|^3),
$$
where we can readily identify the coefficients as 
$$
 A^{ij}_\alpha=\mathbf{e}_{n+\alpha}\cdot \mathbf{A}_{\Sigma'}^\Real|_{\mathbf{0}}(\mathbf{e}_i, \mathbf{e}_j)
$$
where $ \mathbf{A}_{\Sigma'}^\Real$ is the second fundamental form of $\Sigma'$.

\begin{lem}\label{DistLem}
With notation as above, we have the asymptotic expansion
$$
|E(G^{-1}_{M'}(\mathbf{x}))|^2=|\mathbf{x}|^2+\frac{1}{3} Q'(\mathbf{x})+O(|\mathbf{x}|^5)
$$
where $Q$ is a homogeneous degree four polynomial of the form
$$
Q'(\mathbf{x})=\sum_{\alpha=1}^N \left(\sum_{i,k=1}^n C_{\alpha}^{ik} x_i x_k\right)^2+\sum_{i,j,k,l=1}^n  H^{ijkl} x_i x_j x_l x_k
$$
where the $H^{ijkl}$ satisfy 	$
H^{iiii}=0, 1\leq i \leq n
$
and, for $i\neq j$,
$$
H^{ijij}+H^{ijji}+H^{jiji}+H^{jiij}+H^{iijj}+H^{jjii}=0.
$$
\end{lem}
\begin{proof}
The expansion of the metric in geodesic normal coordinates yields
	$$
	g^E_{ij}=(E^* g')|_{\mathbf{x}}(\mathbf{e}_i,\mathbf{e}_j)=\delta_{ij} -\frac{1}{3}\sum_{k,l=1}^n R_{ikjl}x_k x_l+O(|\mathbf{x}|^3),
	$$
	where here $
	R_{ijkl}=Riem|_{p}(\mathbf{e}_i, \mathbf{e}_j, \mathbf{e}_k, \mathbf{e}_l).$
	are the coefficients the Riemann curvature tensor at $p$.
	Likewise, 
	$$
	g^G_{ij}=(G^*_{M'} g')|_{\mathbf{x}}(\mathbf{e}_i, \mathbf{e}_j)= \delta_{ij} +\sum_{k,l=1}^n B^{ijkl}x_k x_l+O(|\mathbf{x}|^3)
	$$
	where, by  \cite[pg. 89]{karpVolumeSmallExtrinsic1989}, we have
	$$
	B^{ijkl}=\sum_{\alpha=1}^N C_\alpha^{il}C_\alpha^{jk}.
	$$
%
	As $E(\mathbf{0})=E_{M'}(\mathbf{0})$ and $DE(\mathbf{0})=DG_{M'}(\mathbf{0})$ and
	$(E^* g')_{ij}-	(G^* g_{\Sigma})_{ij}=O(|\mathbf{x}|^2)$
	it readily follows that
	$$
	E^{-1}(G_{M'}(\mathbf{x}))=(x_1+K_1(\mathbf{x})), \ldots, x_n+K_n(\mathbf{x}))+O(|\mathbf{x}|^4)
	$$
	where $K_i$ are cubic homogeneous polynomials. In fact, for the metrics to agree to second order,  one must have, for $1\leq i, j\leq n$,
	$$
	\partial_{j} K_i(\mathbf{x}) +\partial_{i} K_j(\mathbf{x}) =  \sum_{k,l=1}^n (B^{ij kl} x_k x_l+\frac{1}{3}R_{ikjl} x_k x_l).
	$$

Using $K_i(\mathbf{x}) =\frac{1}{3} \sum_{j=1}^n x_j \partial_j K_i(\mathbf{x})$, we obtain
	\begin{align*}
		|E^{-1}(G_{M'}(\mathbf{x}))|^2&-|\mathbf{x}|^2 =2\sum_{i=1}^n x_i K_i(\mathbf{x}) +O(|\mathbf{x}|^5)\\
		&= \frac{2}{3}\sum_{i,j=1}^n x_i x_j \partial_j K_i(\mathbf{x})+O(|\mathbf{x}|^5)\\
		&=\frac{1}{3}\sum_{i, j=1}^n x_i x_j(\partial_i K_j(\mathbf{x})+\partial_j K_i(\mathbf{x}))+O(|\mathbf{x}|^5)\\
		&=\frac{1}{3}\sum_{i, j,k,l=1}^n (B^{ijkl}+\frac{1}{3} R_{ikjl}) x_ix_j x_k x_l  +O(|\mathbf{x}|^5).
	\end{align*}
To conclude we observe
$$
\sum_{i, j,k,l=1}^n B^{ijkl}  x_ix_j x_k x_l =\sum_{\alpha=1}^N \left(\sum_{j,l=1}^n C^{ik}_\alpha x_i x_j\right)^2.
$$
While if we set
	$$
	H^{ikjl}= \frac{1}{3} R_{ikjl},
	$$
 then the algebraic symmetries of the Riemann curvature tensor imply that,  for $1\leq i \leq n$, $
	H^{iiii}=H^{iijj}=H^{jjii}=0$
	and, for $i\neq j$, 
	\begin{align*}
		0&=\frac{1}{3}\left(R_{ijij}+ R_{ijji}+R_{jiij}+R_{jiji}\right)=H^{ijij}+H^{ijji}+H^{jiij}+H^{jiji}\\
		&=H^{ijij}+H^{ijji}+H^{jiji}+H^{jiij}+H^{iijj}+H^{jjii}.
	\end{align*}
The claim follows.	
\end{proof}

We are now ready to prove the extension of the Karp-Pinsky estimate. 
\begin{proof}
   Continuing with the notation from above, let $\Sigma''=G_{M'}^{-1}(\Sigma')\subset \Real^{n+K}\times \set{0}$.  Up to shrinking $U$, this is a submanifold in $\Real^{n+K}$ through $\mathbf{0}$ and tangent to $\Real^n\times\set{0}$ at that point.  In particular, up to shrinking $V$, we can express $\Sigma''$ as a graph of a function $\mathbf{z}(\mathbf{w})=(z_1(\mathbf{w}), \ldots, z_{K}(\mathbf{w}))$ and can define a map
   $$
   G_{\Sigma''}: V\to \Sigma'', \mathbf{w}\mapsto (\mathbf{w}, \mathbf{z}(\mathbf{w})).
   $$
   Note the hypotheses on $\Sigma'$ ensure $\mathbf{v}(\mathbf{0})=\mathbf{0}$ and $D\mathbf{v}(\mathbf{0})=\mathbf{0}$.  In particular, we have
   $$
   z_\alpha(\mathbf{w})=\frac{1}{2} \sum_{i,j=1}^{n} \hat{A}^{ij}_\alpha w_i w_j+O(|\mathbf{w}|^3),
   $$
   where we can readily identify these terms with 
   $$
   \hat{A}^{ij}_\alpha=\mathbf{e}_{n+\alpha}\cdot \mathbf{A}_{\Sigma''}^\Real|_{\mathbf{0}}(\mathbf{e}_i, \mathbf{e}_j).
   $$
   As $
   G_{\Sigma'}=G_{M'}\circ G_{\Sigma''}$, the identification of $T_p M$ and $T_{\mathbf{0}} M'$ implies
   $$
   A_{\alpha}^{ij}=\hat{A}_{\alpha}^{ij}=\mathbf{e}_{\alpha+n}\cdot \mathbf{A}^T_{\Sigma'}|_{\mathbf{0}}(\mathbf{e}_i, \mathbf{e}_j)=g(\mathbf{e}_{\alpha+n}, \mathbf{A}^g_{\Sigma}|_{p}(\mathbf{e}_i, \mathbf{e}_j)) , 1\leq \alpha \leq K,$$
   where here $\mathbf{A}_{\Sigma}^g$ is the second fundamental form of $\Sigma$ in $(M,g)$ and $\mathbf{A}^T_{\Sigma'}$ the projection of the second fundamental form of $\Sigma'$ onto $TM'$.
   Likewise,
   $$
   A_{\alpha+K}^{ij}=C_{\alpha}^{ij}=\mathbf{e}_{\alpha+n+K}\cdot \mathbf{A}^{N}_{\Sigma'}|_{\mathbf{0}}(\mathbf{e}_i, \mathbf{e}_j), 1\leq \alpha \leq N,
   $$
   where here $\mathbf{A}^{N}_{\Sigma'}$ is the projection of the second fundamental form of $\Sigma'$ onto $NM'$.

     Arguing as in \cite{karpVolumeSmallExtrinsic1989}, there is a homogeneous quartic polynomial $Q$ so 
   $$
    |G_{\Sigma''}(\mathbf{w})|^2 =|\mathbf{w}|^2 + \frac{1}{4}Q(\mathbf{w})+ O(|\mathbf{w}|^5)
   $$
   where here 
   $$
   Q(\mathbf{w})= \sum_{i,j,k,l=1}^n \sum_{\alpha=1}^K \hat{A}^{ik}_\alpha \hat{A}^{jl}_\alpha w_i w_j w_k w_l=\sum_{i,j,k,l=1}^n \sum_{\alpha=1}^K {A}^{ik}_\alpha {A}^{jl}_\alpha w_i w_j w_k w_l.
   $$
  
   By Proposition \ref{DistLem}, if $r(q)$ is the geodesic distance in $M'$ between $\mathbf{q}$ and $\mathbf{0}$, then
   $$
   (r(G_{M'}(\mathbf{x})))^2=|\mathbf{x}|^2+\frac{1}{3} Q'(\mathbf{x})+ O(|\mathbf{x}|^5).
   $$
   Combining the two expansions yields,
   $$
   (r(G_{\Sigma'}(\mathbf{w})))^2=|\mathbf{w}|^2+\frac{1}{3} Q'((\mathbf{w},\mathbf{0}))+\frac{1}{4} Q(\mathbf{w})+ O(|\mathbf{w}|^5).
   $$
   In particular, when $R>0$ is sufficiently small there is a function $h$ so that
   $$
   G_{\Sigma'}^{-1}(\mathcal{B}^{g'}_{R}(p))=\set{\mathbf{w}: h(\mathbf{w})\leq r}
   $$
   where here $h$ satisfies $h(\mathbf{0})=0$, and, for $\mathbf{w}\neq \mathbf{0}$, one has 
   $$
   h(\mathbf{w})=|\mathbf{w}|-\frac{1}{6} Q'((\hat{\mathbf{w}},\mathbf{0}))|\mathbf{w}|^3-\frac{1}{8} Q(\hat{\mathbf{w}})|\mathbf{w}|^3+O(|\mathbf{w}|^4)
   $$
   where $\hat{\mathbf{w}}=|\mathbf{w}|^{-1} \mathbf{w}$.
   
   Following the argument of Karp-Pinksy directly, one obtains, for $R$ small, 
   \begin{align*}
   	|\Sigma\cap \mathcal{B}_R^g(p)|_g&= |B_1^n|_{\Real} R^n+A|B_1^n|_{\Real} R^{n+2} + O(R^{n+3}),
   \end{align*}
where the subleading coefficient is given as
$$
A =\frac{1}{2|B_1^n|_{\Real}}  \left( \frac{|B_1^n|_{\Real}}{n+2} \sum_{\alpha=1}^{N+K}\sum_{i,k=1}^n |A_\alpha^{ik}|^2 -\frac{1}{3}\int_{\mathbb{S}^{n-1}} Q'-\frac{1}{4}\int_{\mathbb{S}^{n-1}} Q\right). 
$$
It is shown in \cite[pg. 90]{karpVolumeSmallExtrinsic1989} that
$$
\int_{\mathbb{S}^{n-1}} Q= \frac{ |B_1^n|_{\Real}}{n+2} \sum_{\alpha=1}^K \left(2\sum_{i,j=1}^n |A^{ij}_\alpha|^2 + \left( \sum_{i=1}^n A^{ii}_\alpha\right)^2 \right).  
$$
The exact same computation and the properties of the $H^{ijkl}$ also imply that
\begin{align*}
\int_{\mathbb{S}^{n-1}} Q'&=\frac{ |B_1^n|_{\Real}}{n+2} \sum_{\alpha=1}^N  \left(2\sum_{i,j=1}^n |C^{ij}_\alpha|^2 + \left( \sum_{i=1}^n C^{ii}_\alpha\right)^2 \right)\\
&=\frac{ |B_1^n|_{\Real}}{n+2} \sum_{\alpha=1}^N  \left(2\sum_{i,j=1}^n |A^{ij}_{\alpha+K}|^2 + \left( \sum_{i=1}^n A^{ii}_{\alpha+K}\right)^2 \right).
\end{align*}

Finally, using the identification of the $A_{\alpha}^{ij}$ with the geometric data we have,
$$
 \sum_{\alpha=1}^{N+K}\sum_{i,k=1}^n |A_\alpha^{ik}|^2 =|\mathbf{A}_{\Sigma'}^{\Real}(\mathbf{0})|^2
 $$
Likewise,  using $G_{M'}^* g'=g_{\Real}+O(|\mathbf{x}|^2)$, yields
$$
 \sum_{\alpha=1}^K\sum_{i,j=1}^n |A^{ij}_\alpha|^2 =|\mathbf{A}_{\Sigma}^g(p)|^2 \mbox{ and } \sum_{\alpha=1}^K \left(\sum_{i=1}^n A^{ii}_\alpha\right)^2 =|\mathbf{H}_{\Sigma}^g(p)|^2
$$
and
$$
\sum_{\alpha=1}^N\sum_{i,j=1}^n |A^{ij}_{\alpha+K}|^2 =|\mathbf{A}_{\Sigma'}^N(\mathbf{0})|^2 \mbox{ and }\sum_{\alpha=1}^N \left(\sum_{i=1}^n A^{ii}_{\alpha+K}\right)^2 =|\mathbf{H}_{\Sigma'}^N(\mathbf{0})|^2.
$$
It is clear that at $i(p)=\mathbf{0}$
$$
|\mathbf{A}_{\Sigma'}^{\Real}(\mathbf{0})|^2= |\mathbf{A}_{\Sigma}^g(p)|^2+|\mathbf{A}_{\Sigma'}^N(\mathbf{0})|^2 \mbox{ and } |\mathbf{H}_{\Sigma'}^{\Real}(\mathbf{0})|^2= |\mathbf{H}_{\Sigma}^g(p)|^2+|\mathbf{H}_{\Sigma'}^N(\mathbf{0})|^2.
$$
The Gauss equations imply
$$
-R_{\Sigma}^g(p)=-R_{\Sigma'}^{\Real}(\mathbf{0})= |\mathbf{A}_{\Sigma'}^{\Real}(\mathbf{0})|^2-|\mathbf{H}_{\Sigma'}^{\Real}(\mathbf{0})|^2.
$$
Hence, the coefficient $A$  can be expressed as
\begin{align*} 
	A&= \frac{1}{2(n+2)}\left( |\mathbf{A}_{\Sigma}^\Real|^2-\frac{2}{3} |\mathbf{A}_{\Sigma'}^N|^2-\frac{1}{3} |\mathbf{H}_{\Sigma'}^N|^2 -\frac{1}{2} |\mathbf{A}_\Sigma^g|^2 -\frac{1}{4} |\mathbf{H}_{\Sigma}^g|^2\right)\\
	&= \frac{1}{2(n+2)}\left( |\mathbf{A}_{\Sigma}^g|^2+\frac{1}{3} |\mathbf{A}_{\Sigma'}^N|^2-\frac{1}{3} |\mathbf{H}_{\Sigma'}^N|^2 -\frac{1}{2} |\mathbf{A}_\Sigma^g|^2 -\frac{1}{4} |\mathbf{H}_{\Sigma}^g|^2\right)\\
	&= \frac{1}{2(n+2)}\left(\frac{1}{3} |\mathbf{A}_{\Sigma'}^\Real|^2-\frac{1}{3} |\mathbf{H}_{\Sigma'}^\Real|^2 +\frac{1}{6} |\mathbf{A}_\Sigma^g|^2 +\frac{1}{12} |\mathbf{H}_{\Sigma}^g|^2\right)\\
	&= \frac{1}{6(n+2)}\left( -R_{\Sigma}^g+\frac{1}{2} |\mathbf{A}_\Sigma^g|^2 +\frac{1}{4} |\mathbf{H}_{\Sigma}^g|^2\right).\
\end{align*}
This concludes the proof.
\end{proof}
\bibliographystyle{hamsabbrv}
\bibliography{Library2}
\end{document}